\documentclass[12pt]{amsart}
\usepackage{amsmath,amsthm,amsfonts,amssymb,eucal}

\renewcommand {\a}{ \alpha }
\renewcommand{\b}{\beta}
\newcommand{\e}{\epsilon}

\newcommand{\g}{\gamma}

\newcommand{\vark}{\varkappa}
\renewcommand{\d}{\delta}
\newcommand{\s}{\sigma}
\renewcommand{\l}{\lambda}

\newcommand{\z}{\zeta}
\renewcommand{\t}{\theta}

\newcommand{\T}{\Theta}
\newcommand{\p}{\partial}

\newcommand{\R}{ \mathbb R}

\newcommand {\BA}{\mathbf A}
\newcommand {\BB}{\mathbf B}

\newcommand {\BE}{\mathbf E}

\newcommand {\BK}{\mathbf K}

\newcommand {\BS}{\mathbf S}
\newcommand {\BR}{\mathbf R}

\newcommand {\BT}{\mathbf T}

\newcommand {\bx}{\mathbf x}

\newcommand {\by}{\mathbf y}

\newcommand{\SM}{{\sf{M}}}

\newcommand {\bxi}{\boldsymbol\xi}

\newcommand{\lu}{\langle}
\newcommand{\ru}{\rangle}

\newcommand{\plainC}[1]{\textup{{\textsf{C}}}^{#1}}

\newcommand{\plainH}[1]{\textup{{\textsf{H}}}^{#1}}

\newcommand{\plainL}[1]{\textup{{\textsf{L}}}^{#1}}

\DeclareMathOperator {\re} {{Re}}

\DeclareMathOperator{\op}{{Op}}

\newtheorem{thm}{Theorem}
\newtheorem{cor}[thm]{Corollary}
\newtheorem{lem}[thm]{Lemma}
\newtheorem{prop}[thm]{Proposition}

\theoremstyle{definition}

\newtheorem{con}[thm]{Conjecture}

\begin{document}
\hoffset -4pc

\title
[A family of anisotropic integral operators]
{{A family of anisotropic integral operators and behaviour of its maximal
eigenvalue}}
\author[B.S. Mityagin, A.V. Sobolev]{B.S. Mityagin, A.V. Sobolev}
\address{Department of Mathematics\\ The Ohio State University\\
231 West 18th Ave\\ Columbus\\ OH 43210 USA}
\email{mityagin.1@osu.edu}
 \address{Department of Mathematics\\ University College London\\
Gower Street\\ London\\ WC1E 6BT UK}
\email{asobolev@math.ucl.ac.uk}

\begin{abstract}
We study the family of compact integral operators $\BK_\b$ in $\plainL2(\R)$
with the kernel
\begin{equation*}
K_\b(x, y) = \frac{1}{\pi}\frac{1}{1 + (x-y)^2 + \b^2\Theta(x, y)},
\end{equation*}
depending on the parameter $\b >0$, where $\T(x, y)$ is a symmetric non-negative
homogeneous function of degree $\g\ge 1$.
The main result
%
is the following asymptotic formula for the maximal eigenvalue $\SM_\b$
of $\BK_\b$:
\begin{equation*}
\SM_\b = 1 - \l_1 \b^{\frac{2}{\g+1}} + o(\b^{\frac{2}{\g+1}}), \b\to 0,
\end{equation*}
where $\l_1$ is the lowest eigenvalue of the operator $\BA = |d/dx| + \frac{1}{2}\T(x, x)$.
A central role in the proof is played
by the fact that $\BK_\b, \b>0,$ is positivity improving.
The case $\T(x, y) = (x^2 + y^2)^2$ has been studied earlier in
the literature as a simplified model of high-temperature superconductivity.

\end{abstract}

\maketitle

\section{Introduction and the main result}

\subsection{Introduction}
The object of the study is the following family of integral operators on $\plainL2(\R)$:
\begin{equation}\label{bkb:eq}
\BK_\b u(x) = \int  K_\b(x, y) u(y) dy,
\end{equation}
(here and below we omit the domain of integration if it
is the entire real line $\R$) with the kernel
\begin{equation}\label{kernel:eq}
K_\b(x, y) = \frac{1}{\pi}\frac{1}{1 + (x-y)^2 + \b^2\Theta(x, y)},
\end{equation}
where $\b>0$ is a small parameter, and the function
$\Theta = \Theta(x, y)$ is a homogeneous
non-negative function of $x$ and $y$ such that
\begin{equation}\label{Theta:eq}
\Theta(t x, t y) = t^\g \Theta(x, y), \ \g >0,
\end{equation}
for all $x, y\in\R$ and $t >0$,
and the following conditions are satisfied:
\begin{equation}\label{Theta1:eq}
\begin{cases}
c\le \Theta(x, y)\le C, \ \  |x|^2+|y|^2 = 1,\\[0.2cm]
\Theta(x, y) = \Theta(y, x),  x, y\in\R.
\end{cases}
\end{equation}
By $C$ or $c$ (with or without indices) we denote various positive
constants whose value is of no importance.
The conditions \eqref{Theta:eq} and \eqref{Theta1:eq} guarantee that
the operator $\BK_\b$ is self-adjoint and compact.

Such an  operator, with $\Theta(x, y) = (x^2+ y^2)^2$ was suggested by P. Krotkov and A. Chubukov
in \cite{KC1} and \cite{KC2} as a simplified model of
high-temperature superconductivity. The analysis in \cite{KC1}, \cite{KC2} reduces
to the asymptotics of the top eigenvalue $\SM_\b$ of the operator $\BK_\b$ as
$\b\to 0$. Heuristics in \cite{KC1} and \cite{KC2} suggest that
$\SM_\b$ should behave as $1 - w \b^{\frac{2}{5}}+o(\b^{\frac{2}{5}})$ with
some positive constant $w$. A mathematically rigorous argument given by B. S. Mityagin
in \cite{M} produced a
two-sided bound supporting this formula. The aim of the present paper is
to find and justify an appropriate two-term asymptotic formula for
$\SM_\b$ as $\b\to 0$ for
a homogeneous function $\Theta$ satisfying \eqref{Theta:eq}, \eqref{Theta1:eq}
and some additional smoothness conditions (see \eqref{lip:eq}).

As $\b\to 0$, the operator
$\BK_\b$ converges strongly to the positive-definite operator
$\BK_0$, which is no longer compact. The norm of $\BK_0$ is easily
found using the Fourier transform
\begin{equation*}
\hat f(\xi) = \frac{1}{\sqrt{2\pi}} \int e^{-i\xi x} f(x) dx,
\end{equation*}
which is unitary on $\plainL2(\R)$.
Then one checks directly that
\begin{equation}\label{m:eq}
\textup{the Fourier transform of}\ \ \ m_t(x) = \frac{t}{\pi}\frac{1}{t^2+x^2},\  t >0,
\ \ \ \textup{equals}\ \ \ \hat m_t(\xi) = \frac{1}{\sqrt{2\pi}}e^{-t|\xi|},
\end{equation}
and hence the operator $\BK_0$ is unitarily equivalent to the multiplication by
the function $e^{-|\xi|}$, which means that
$\|\BK_0\| = 1$.

\subsection{The main result} For the maximal eigenvalue $\SM_\b$ of the operator
$\BK_\b$ denote by $\Psi_\b$ the corresponding normalized eigenfunction.
Note that the operator
$\BK_\b$ is positivity improving, i.e. for any non-negative non-zero function $u$ the function
$\BK_\b u$ is positive a.a. $x\in\R$ (see \cite{RS4}, Chapter XIII.12). Thus, by \cite{RS4},
Theorem XIII.43 (or by \cite{Davies}, Theorem 13.3.6),
the eigenvalue $\SM_\b$ is non-degenerate and the eigenfunction
$\Psi_\b$ can be assumed to be positive a.a. $x\in\R$. From now on we always choose
$\Psi_\b$ in this way.
The behaviour of $\SM_\b$ as $\b\to 0$, is governed by the model operator
\begin{equation}\label{modela:eq}
(\BA u)(x) = |D_x|u(x) + 2^{-1}\t(x) u(x),
\end{equation}
where
\begin{equation*}
\t(x) = \T(x, x)
=
\begin{cases}
|x|^\g \T(1, 1),\ x \ge 0;\\[0.2cm]
|x|^\g \T(-1, -1), x <0.
\end{cases}
\end{equation*}
This operator is understood as the pseudo-differential operator $\op(a)$ with the
symbol
\begin{equation}\label{symbola:eq}
a(x, \xi) = |\xi| + 2^{-1}\t(x).
\end{equation}
For the sake of completeness recall that $P=\op(p)$ is a pseudo-differential
operator with the symbol $p = p(x, \xi)$ if
\begin{equation*}
(Pu)(x) = \frac{1}{2\pi}\int\int e^{i(x-y)\xi} p(x, \xi) u(y) dy d\xi
\end{equation*}
for any Schwartz class function $u$.
The operator $\BA$ is essentially self-adjoint on
$\plainC\infty_0(\R)$, and has a purely discrete
spectrum (see e.g. \cite{Sch}, Theorems 26.2, 26.3). Using the von Neumann Theorem
(see e.g. \cite{RS2}, Theorem X.25), one can see that $\BA$ is self-adjoint
on $D(\BA) = D(|D_x|)\cap D(|x|^{\g })$,
i.e. $D(\BA) = \plainH1(\R)\cap\plainL2(\R, |x|^{2\g})$.
Denote by $\l_l>0$, $l = 1, 2, \dots$ the eigenvalues of $\BA$ arranged in ascending order,
and by $\phi_l$ -- a set of corresponding normalized eigenfunctions.
As shown in Lemma \ref{positivity:lem}, the lowest eigenvalue $\l_1$ is
non-degenerate and its eigenfunction $\phi_1$
can be chosen to be non-negative a.a. $x\in \R$.
From now on we always choose $\phi_1$ in this way.

The main result of
this paper is contained in the next theorem.

\begin{thm}\label{main:thm}
Let $\BK_\b$ be an integral operator defined by
\eqref{bkb:eq} with $\g \geq 1$. Suppose that
the function $\T$ satisfies conditions \eqref{Theta:eq}, \eqref{Theta1:eq}
and the following Lipshitz conditions:
\begin{equation}\label{lip:eq}
\begin{cases}
|\T(t, 1) - \T(1, 1)|\le C|t-1|, \ t\in (1-\e, 1+\e),\\[0.2cm]
|\T(t, -1) - \T(-1, -1)|\le C|t+1|, \  t\in (-1-\e, -1 +\e),
\end{cases}
\end{equation}
with some $\e >0$.
Let $\SM_\b$ be the largest eigenvalue of the operator $\BK_\b$
and $\Psi_\b$ be the corresponding eigenfunction. Then
\begin{equation*}
\lim_{\b\to 0}\b^{-\frac{2}{\g+1}} (1-\SM_{\b}) = \l_1.
\end{equation*}
Moreover, the rescaled eigenfunctions
 $\a^{-\frac{1}{2}}\Psi_\b(\a^{-1}\ \cdot\ ),\ \a = \b^{\frac{2}{\g+1}}$,
converge in norm to $\phi_1$ as $\b\to 0$.
\end{thm}

The top eigenvalue of $\BK_\b$ was studied by B. Mityagin in \cite{M} for
$\Theta(x, y) = (x^2+y^2)^\s$, $\s>0$.
It was conjectured that $\lim_{\b\to0}\b^{-{\frac{2}{2\s+1}}}(1-\SM_\b) = L$ with some $L>0$, but
only the  two-sided bound
\[
c\b^{\frac{2}{2\s+1}}\le 1 - \SM_\b \le C\b^{\frac{2}{2\s+1}},
\]
with some constants $0< c\le C$ was proved.
It was also conjectured that
in the case $\s = 2$ the constant $L$ should coincide with
the lowest eigenvalue of
the operator $|D_x|+ 4x^4$.
Note that for this case the corresponding operator \eqref{modela:eq}
is in fact $|D_x| + 2x^4$.
J. Adduci found an approximate numerical value $\lambda_1 = 0.978...$ in this case,
see \cite{Add}.

Similar eigenvalue asymptotics were investigated by H. Widom in \cite{Widom}
for integral operators with difference kernels. Some ideas of this paper are used in the proof
of Theorem \ref{main:thm}.

Let us now establish the non-degeneracy of the eigenvalue $\l_1$.

\begin{lem} \label{positivity:lem} Let $\BA$ be as defined in \eqref{modela:eq}. Then
\begin{enumerate}
\item
The semigroup $e^{-t\BA}$ is positivity improving for all $t>0$,
\item
The lowest  eigenvalue $\l_1$ is non-degenerate, and
the corresponding eigenfunction $\phi_1$ can be chosen to be positive a.a. $x\in \R$.
\end{enumerate}
\end{lem}

\begin{proof}  The non-degeneracy of $\l_1$ and positivity of the eigenfunction $\phi_1$
would follow from the fact that $e^{-t\BA}$ is positivity improving for all $t>0$, see
\cite{RS4}, Theorem XIII.44. The proof of this fact
is done by comparing the semigroups for the operators $\BA$ and $\BA_0 = |D_x|$.
Using \eqref{m:eq} it is straightforward to find the integral kernel of $e^{-t\BA_0}$:
\[
m_t(x-y) = \frac{1}{\pi}\frac{t}{t^2+(x-y)^2}, t >0,
\]
which shows that $e^{-t\BA_0}$ is positivity improving. To extend the same conclusion
to $e^{-t\BA}$ let
\[
V_n(x)
=
\begin{cases}
2^{-1}\t(x),\ |x|\le n,\\[0.2cm]
2^{-1}\t(\pm n), \pm x > n,
\end{cases}
n = 1, 2, \dots.
\]
Since $(\BA_0+V_n)f\to \BA f$ and $(\BA-V_n)f\to \BA_0 f$ as $n\to\infty$
for any $f\in\plainC\infty_0(\R)$, by \cite{RS1}, Theorem VIII.25a the operators
$\BA_0 + V_n$ and $\BA - V_n$ converge
to $\BA$ and $\BA_0$ resp. in the strong resolvent sense as $n\to\infty$. Thus by
\cite{RS4}, Theorem XIII.45, the semigroup $e^{-t\BA}$ is also positivity improving
for all $t >0$, as required.
\end{proof}

\subsection{Rescaling}
As a rule, instead of $\BK_\b$ it is more convenient to work with
the operator obtained by rescaling $x\to \a^{-1}x$ with $\a >0$. Precisely,
let $U_\a$ be the unitary operator on $\plainL2(\R)$ defined as
$(U_{\a}f)(x) = \a^{-\frac{1}{2}} f(\a^{-1} x)$. Then
 $U_\a \BK_\b U_\a^*$ is the integral operator with
the kernel
\begin{equation*}
\frac{\a}{\pi}\frac{1}{\a^2 + (x-y)^2 + \b^2 \a^{-\g+2}\Theta(x, y)}.
\end{equation*}
Under the assumption $\b^2 = \a^{\g+1}$, this kernel becomes
\begin{equation}\label{B:eq}
B_\a(x, y) = \frac{\a}{\pi} \frac{1}{\a^2+(x-y)^2 + \a^3\Theta(x, y)}.
\end{equation}
Thus, denoting the corresponding integral operator by $\BB_\a$, we get
\begin{equation}\label{alpha:eq}
\BK_\b = U_\a^* \BB_\a U_\a,\ \a = \b^{\frac{2}{\g+1}}.
\end{equation}
Henceforth the value of $\a$ is always chosen as in this formula.

Denote by $\mu_\a$
the maximal eigenvalue of the operator
$\BB_{\a}$, and by $\psi_\a$
-- the corresponding normalized eigenfunction.
By the same token as for the operator
$\BK_\b$,
the eigenvalue $\mu_\a$ is non-degenerate and the choice of the corresponding eigenfunction
$\psi_\a$ is determined uniquely by the requirement that $\psi_\a > 0$ a.e.. Moreover,
\begin{equation}\label{unitary:eq}
\mu_\a = \SM_{\b}, \ \psi_\a(x) = (U_\a \Psi_\b)(x) = \a^{-\frac{1}{2}}\Psi_\b(\a^{-1} x), \
\a = \b^{\frac{2}{\g+1}}.
\end{equation}
This rescaling allows one to rewrite Theorem \ref{main:thm} in a somewhat more
compact form:

\begin{thm}\label{main_b:thm} Let $\g\ge 1$ and suppose that
the function $\T$ satisfies conditions \eqref{Theta:eq}, \eqref{Theta1:eq}
and \eqref{lip:eq}. Then
\begin{equation*}
\lim_{\a\to 0}\a^{-1} (1-\mu_\a) = \l_1.
\end{equation*}
Moreover, the eigenfunctions $\psi_\a$,
converge in norm to $\phi_1$ as $\a\to 0$.
\end{thm}

The rest of the paper is devoted to the proof of Theorem \ref{main_b:thm}, which immediately
implies Theorem \ref{main:thm}.

\section{``De-symmetrization" of $\BK_\b$ and $\BB_{\a}$}

First we de-symmetrize the operator $\BK_\b$.
Denote
\begin{equation*}
\BK^{(l)}_\b u(x) = \int K^{(l)}_\b(x, y) u(y) dy,
\end{equation*}
with the kernel
\begin{equation*}
K^{(l)}_\b(x, y) = \frac{1}{\pi}\frac{1}{1 + (x-y)^2 + \b^2 \t(x)}.
\end{equation*}

\begin{lem} Let $\b \le 1$ and $\g \ge 1$.
Suppose that the conditions \eqref{Theta:eq}, \eqref{Theta1:eq} and \eqref{lip:eq} are satisfied.
Then
\begin{equation}\label{kbeta:eq}
\|\BK^{(l)}_\b - \BK_\b\|\le C_q \b^{\frac{2}{\g}}.
\end{equation}
\end{lem}

\begin{proof}
Due to \eqref{Theta:eq} and \eqref{Theta1:eq},
\begin{equation}\label{both:eq}
c(|t| + 1)^\g\le \T(t, \pm 1)\le C(|t| + 1)^\g, \ \ t\in\R.
\end{equation}
Also,
\begin{equation}\label{lip1:eq}
\begin{cases}
|\Theta(t, 1) - \Theta(1, 1)| \le  C  (|t|+1)^{\g-1} |t - 1|,\\[0.2cm]
|\Theta(t, -1) - \Theta(-1, -1)| \le  C  (|t|+1)^{\g-1} |t + 1|,
\end{cases}
\end{equation}
for all $t\in\R$. Indeed,   \eqref{lip:eq} leads to the first inequality \eqref{lip1:eq}
for $|t-1|<\e$. For $|t-1|\ge \e$ it  follows from \eqref{both:eq} that
\[
|\T(t, 1) - \T(1, 1)|\le C(|t|+1)^\g\le C' \e^{-1}(|t|+1)^{\g-1}|t-1|.
\]
The second bound in \eqref{lip1:eq} is checked similarly.

Now we can estimate the difference of the kernels
\begin{align}\label{kdiff:eq}
K_\b(x, y) - &\ K^{(l)}_\b(x, y)\notag\\[0.2cm]
= &\ \frac{1}{\pi} \frac{\b^2 \bigl(\Theta(x, x) - \Theta(x, y)\bigr)}
{\bigl(1+(x-y)^2 + \b^2 \T(x, y)\bigr)
\bigl(1+(x-y)^2 +  \b^2 \T(x, x)\bigr)}.
\end{align}
It follows from \eqref{lip1:eq} with $ t = y |x|^{-1}$ that
\begin{equation*}
|\Theta(x, x) - \Theta(y, x)|
\le C(|x| + |y|)^{\g-1} |x-y|.
\end{equation*}
Substituting into \eqref{kdiff:eq}, we get
\begin{equation*}
|K_\b(x, y) - K^{(l)}_\b(x, y)| \le C
\frac{|x-y|}{(1+(x-y)^2)^{2-\d}}\
\frac{\b^2  \bigl(|x| + |y|\bigr)^{\g-1}}
{(1 + \b^2 (|x| + |y|)^\g)^\d},
\end{equation*}
for any $\d \in (0, 1)$.
The second factor on the right-hand side does not exceed
\begin{equation*}
\b^{\frac{2}{\g}} \max_{t\ge 0}\frac{t^{\g-1}}{(1+ t^\g)^\d},
\end{equation*}
which is bounded by $C\b^{2/\g}$ under the assumption that $\d\ge 1 - \g^{-1}$.
Therefore
\begin{equation*}
|K_\b(x, y) - K^{(l)}_\b(x, y)| \le C\b^{\frac{2}{\g}}
\frac{|x-y|}{(1+(x-y)^2)^{2-\d}}.
\end{equation*}
For any $\d\in (0, 1)$ the right hand side is integrable in $x$ (or $y$). Now,
estimating the norm using the standard Schur Test, see Proposition
\ref{schur:prop}, we conclude that
\begin{equation*}
\|\BK_\b - \BK^{(l)}_\b\|
\le C\b^{\frac{2}{\g}}\int\frac{|t|}{(1+t^2)^{2-\d}}dt
\le C'\b^{\frac{2}{\g }},
\end{equation*}
which is the required bound.
\end{proof}

Similarly to the operator $\BK_\b$, it is readily checked by scaling that
the operator $\BK^{(l)}_\b$ is unitarily equivalent to the
operator $\BB^{(l)}_\a$ with the kernel
\begin{equation}\label{bl_kernel:eq}
B^{(l)}_\a(x, y) = \frac{1}{\pi}\frac{\a}{\a^2 + (x-y)^2 + \a^3 \t(x)}.
\end{equation}
Thus the bound \eqref{kbeta:eq} ensures that
\begin{equation}\label{k-l:eq}
\|\BB_\a - \BB^{(l)}_\a\| = \|\BK_\b - \BK^{(l)}_\b\| \le C\a^{1+\frac{1}{\g}}, \a\le 1,
\end{equation}
see \eqref{alpha:eq} for the definition of $\a$.

\section{Approximation for $\BB^{(l)}_\a$}

\subsection{Symbol of $\BB^{(l)}_\a$}
Now our aim is to show that the operator $I-\a\BA$ is
an approximation of the operator $\BB^{(l)}_\a$,
defined above.
To this end we need to represent $\BB^{(l)}_\a$ as a pseudo-differential operator.
Rewriting the kernel \eqref{bl_kernel:eq} as
\begin{equation*}
B^{(l)}_\a(x, y) = t^{-1} m_{\a t}(x-y),\ t = g_\a(x),
\end{equation*}
with
\begin{equation}\label{galpha:eq}
g_\a(x) = \sqrt{1+\a \t(x)},
\end{equation}
and using \eqref{m:eq}, we can write for any Schwartz class function $u$:
\begin{equation*}
(\BB^{(l)}_\a u)(x) = \frac{1}{2\pi} \int\int
e^{i(x-y)\xi} b^{(l)}_\a(x, \xi) u(y) dy d\xi,\
\end{equation*}
where
\begin{equation*}
b^{(l)}_\a(x, \xi) = \frac{1}{g_\a(x)}
e^{-\a|\xi| g_\a(x)}.
\end{equation*}
Thus $\BB^{(l)}_\a = \op(b^{(l)}_\a)$.

\subsection{Approximation for $\BB^{(l)}_\a$}
Let the operator $\BA$ and
the symbol $a(x, \xi)$ be as defined in \eqref{modela:eq} and
\eqref{symbola:eq}.
Our first objective is to check that the error
\begin{equation*}
r_\a(x, \xi) := b^{(l)}_\a(x, \xi) - (1 - \a a(x, \xi))
\end{equation*}
is small in a certain sense. The condition $\g \ge 1$ will allow
us to use standard norm estimates for pseudo-differential
operators.
Using the formula
\begin{equation*}
e^{-\a y} = 1 - \a y + \a \int_0^y (1-e^{-\a t}) dt, \ y >0,
\end{equation*}
 we can split the error as follows:
\begin{align*}
r_\a(x, \xi) = &\ r^{(1)}_\a(x) + r^{(2)}_\a(x, \xi),\\[0.2cm]
r^{(1)}_\a(x) = &\ \frac{1}{g(x)} + \a 2^{-1}\t(x) - 1,\\[0.2cm]
r^{(2)}_\a(x, \xi) = &\ \frac{\a}{g(x)} \int_0^{|\xi| g(x)}
(1 - e^{-\a t} )dt,
\end{align*}
where we have used the notation $g(x) = g_\a(x)$ with $g_\a$ defined in
\eqref{galpha:eq}. Since $\g\ge 1$, we have
\begin{equation}\label{g_dash:eq}
|g'(x)|\le C g(x), \ C = C(\g), \ x\not = 0,
\end{equation}
for all $\a\le 1$.
Introduce also the function $\z\in\plainC\infty(\R_+)$
such that
\begin{equation*}
\z'(x) \ge 0,\
\z(x) =
\begin{cases}
x,\ \ 0\le x\le 1;\\
2,\ \ \ x\ge 2.
\end{cases}
\end{equation*}
Note that
\begin{equation}\label{zeta_prod:eq}
\z(x_1 x_2)\le 2\z(x_1) x_2,\ \ x_1\ge 0, x_2\ge 1.
\end{equation}
We study the above components $r^{(1)}$, $r^{(2)}$ separately and
introduce the function
\begin{equation}\label{e1:eq}
e^{(1)}_\a(x) = \frac{1}{ \lu x\ru^{\g} \z(\a\lu x\ru^{\g})} r^{(1)}_\a(x),
\end{equation}
and the symbol
\begin{equation}\label{e2:eq}
e^{(2)}_\a(x, \xi) = g_\a(x)^{-\vark} \bigl(\z\bigl((\a\lu \xi\ru)\bigr)^\vark\
\lu\xi\ru\bigr)^{-1}
r^{(2)}_\a(x, \xi),
\end{equation}
where $\vark \in (0, 1]$ is a fixed number.
To avoid cumbersome notation the dependence of $e^{(2)}_\a$ on $\vark$
is not reflected in the notation.
We denote the operators $\op(r_\a)$ and $\op(e_\a)$ by $\BR_\a$ and $\BE_\a$
respectively (with or without superscripts).

\begin{lem}\label{e1:lem} Let $\g\ge 1$. Then for all $\a>0$,
\begin{equation*}
\| e^{(1)}_\a\|_{\plainL\infty}\le C\a.
\end{equation*}
\end{lem}

 \begin{proof}
 Estimate the function $r^{(1)}_\a$:
\begin{equation*}
|r^{(1)}_\a(x)|\le
\begin{cases}
C \a^2|x|^{2\g}, \    \a \t(x)\le 1/2,\\[0.2cm]
C \a|x|^{\g }, \    \a \t(x)> 1/2,
\end{cases}
\end{equation*}
with a constant $C$ independent of $x$.
The second estimate is immediate, and the first one follows from the
Taylor's formula
\[
\frac{1}{\sqrt{1+t}} = 1 - \frac{t}{2} + O(t^2),\ 0\le t\le \frac{1}{2}.
\]
Thus
\[
|r^{(1)}_\a(x)|\le C\a|x|^{\g} \z\bigl(\a|x|^{\g}\bigr).
\]
This leads to the proclaimed estimate for $e^{(1)}_\a$.
 \end{proof}

\begin{lem}\label{e2:lem}
Let $\g\ge1$. Then for all $\a>0$ and any $\vark\in (0, 1]$,
\begin{equation*}
\| \BE_\a^{(2)}\|\le C_\vark\a.
\end{equation*}
\end{lem}

\begin{proof}
To estimate the norm of $\op(e^{(2)}_\a))$ we use
Proposition \ref{boundedness:prop}.
It is clear that the distributional derivatives $\p_x, \p_\xi, \p_{x}\p_{\xi}$
of the symbol $e^{(2)}_\a(x, \xi)$ exist and are given by
%
\begin{align*}
\p_x r^{(2)}_\a(x, \xi)
= &\ -\frac{\a}{g^2} g'\int_0^{|\xi| g} (1- e^{-\a t}) dt
+ \frac{\a}{g}|\xi| g'(1-e^{-\a |\xi| g}),\\[0.2cm]
\p_\xi r^{(2)}_\a(x, \xi) = &\
\a \ \textup{sign}\  \xi (1-e^{-\a|\xi| g}),\\[0.2cm]
\p_x \p_{\xi} r^{(2)}_\a(x, \xi) = &\ \a^2 \xi g' e^{-\a |\xi| g},
\end{align*}
for all $x\not = 0, \xi\not = 0$.
For any $\vark\in (0, 1]$ the elementary bounds hold:
\begin{align*}
\int_0^{|\xi| g} (1- e^{-\a t}) dt \le &\ |\xi| g\z\bigl((\a|\xi| g)^\vark\bigr)
\le 2 |\xi| g^{1+\vark} \z\bigl((\a|\xi|)^\vark\bigr), \\
|1-e^{-\a|\xi|g}|
\le &\ \z\bigl((\a|\xi|g)^\vark\bigr)
\le 2 g^\vark\ \z\bigl((\a|\xi|)^\vark\bigr),\\
 \a |\xi| g e^{-\a|\xi|g} \le &\ \z\bigl((\a|\xi|g)^\vark\bigr)
 \le 2 g^\vark \z\bigl((\a|\xi|)^\vark\bigr).
\end{align*}
Here we have used \eqref{zeta_prod:eq}.
Thus, in view of \eqref{g_dash:eq},
\begin{equation*}
|r^{(2)}_\a(x, \xi)| + |\p_\xi r^{(2)}_\a(x, \xi)| + |\p_x r^{(2)}_\a(x, \xi) |
\le  C\a\lu\xi\ru  g^\vark\z\bigl((\a|\xi|)^\vark\bigr).
\end{equation*}
Also,
\begin{equation*}
|\p_x \p_\xi r_\a^{(2)}(x, \xi)|
\le \a \frac{|g'|}{g}   \bigl(\a |\xi| g e^{-\a|\xi|g} \bigr)
\le C \a |g|^\vark \z\bigl((\a|\xi|)^\vark\bigr).
\end{equation*}
Now estimate the derivatives of the weights:
\begin{align*}
|\p_x g^{-\vark}| = &\ \vark g^{-\vark-1} g'\le C g^{-\vark},\ x\not = 0,\\[0.2cm]
|\p_\xi \bigl(\lu \xi\ru \z\bigl((\a\lu \xi\ru)^\vark\bigr)\bigr)^{-1}|
\le &\ C\frac{1}{\lu\xi\ru^2 \z\bigl((\a\lu\xi\ru)^\vark\bigr)}, \xi\in\R.
\end{align*}
Thus the symbol $e^{(2)}_\a(x, \xi)$ as well as its derivatives
$\p_x, \p_\xi, \p_x \p_\xi$
are bounded by $C\a$ for all $\a >0$ uniformly in $x, \xi$.
Now the required estimate follows from Proposition \ref{boundedness:prop}.
\end{proof}

We make a useful observation:

\begin{cor}\label{g:cor}
Let $\g\ge 1$ and $\vark\in (0, 1]$. Then for any function $f\in D(\BA)$,
\begin{equation}\label{g1:eq}
\a^{-1}\|\BR^{(1)}_\a f\|\to  0,\ \a\to 0,\
\end{equation}
\begin{equation}\label{g2:eq}
\a^{-1}\| \BE^{(2)}_\a \lu D_x\ru \z\bigl((\a\lu D_x\ru)^\vark\bigr) f\|
\to 0, \a\to 0.
\end{equation}

\end{cor}

\begin{proof}
Rewrite:
\begin{equation}\label{r1:eq}
\|\BR^{(1)}_\a f\| = \|\BE^{(1)}_\a   \lu x\ru^{\g } \z(\a\lu x\ru^{\g})f\|
\le \|\BE^{(1)}_\a \|\   \|\lu x\ru^{\g} \z(\a\lu x\ru^{\g})f\|.
\end{equation}
By Lemma \ref{e1:lem} the norm of $\BE^{(1)}_\a$ on the right-hand side
is bounded by $C\a$. The function $\lu x\ru^{\g} \z(\a\lu x\ru^{\g})f$
tends to zero as $\a\to 0$ a.a. $x\in\R$, and it is uniformly bounded by
the function $\lu x\ru^{\g}|f|$, which belongs to $\plainL2$, since
$f\in D(\BA)$. Thus the second factor in \eqref{r1:eq} tends to zero as $\a\to 0$
by the Dominated Convergence Theorem. This proves \eqref{g1:eq}.

Proof of \eqref{g2:eq}. Estimate:
\begin{equation*}
\| \BE^{(2)}_\a \lu D_x\ru\z\bigl((\a\lu D_x\ru)^\vark\bigr) f\|
\le \| \BE^{(2)}_\a \|\ \|\lu \xi\ru \z\bigl((\a\lu \xi\ru)^\vark\bigr) \hat f\|.
\end{equation*}
By Lemma \ref{e2:lem} the norm of the first factor on
the right-hand side is bounded by $C\a$. The second factor tends to zero as
$\a\to 0$ for the same reason as  in the proof of \eqref{g1:eq}.
\end{proof}

\section{Norm-convergence of the extremal eigenfunction}

Recall that the maximal positive eigenvalue $\mu_\a$
of the operator $\BB_\a$ is non-degenerate, and the corresponding (normalized) eigenfunction
$\psi_{\a}$ is positive a.a. $x\in\R$.

The principal goal of this section is to prove that any infinite subset of the family
$\psi_\a$, $\a \le 1$ contains a norm-convergent sequence.
We begin with an upper bound for $1-\mu_\a$ which will be crucial for our argument.

\begin{lem}\label{upper:lem}
If $\g \ge 1$, then
\begin{equation}\label{upper:eq}
\limsup_{\a\to 0} \a^{-1}(1-\mu_\a) \le \l_1.
\end{equation}
\end{lem}

\begin{proof}
Denote $\phi:=\phi_1$.
By a straightforward variational argument it follows that
\begin{align*}
\mu_\a\ge &\ (\BB_\a \phi, \phi)
\ge |(\BB^{(l)}_\a\phi, \phi)| - \|\BB_\a-\BB^{(l)}_\a\|\\[0.2cm]
\ge &\ ((I-\a \BA)\phi, \phi) - |(\BR_\a\phi, \phi)| + o(\a)\\
= &\ 1-\a \l_1 - |(\BR_\a\phi, \phi)| + o(\a),
\end{align*}
where we have also used \eqref{k-l:eq}.
By definitions \eqref{e1:eq} and \eqref{e2:eq},
\begin{equation*}
|\bigl(\BR_\a\phi, \phi\bigr)|
\le \| \BR_\a^{(1)} \phi\| +
\| \BE_\a^{(2)}\lu D_x\ru\z\bigl((\a\lu D_x\ru)^\vark\bigr)\phi\| \ \| g_\a^\vark\phi\|,
\end{equation*}
where $\vark\in (0, 1]$.
It is clear that $g_\a^\vark \phi\in\plainL2$ and its norm is bounded
uniformly in $\a\le1$. The remaining terms on the right-hand side
are of order $o(\a)$ due to Corollary
\ref{g:cor}. This leads to \eqref{upper:eq}.
\end{proof}

The established upper bound leads to the following property.

\begin{lem}\label{psi:lem}
For any $\vark\in (0, 1)$,
\begin{equation*}
\| g_\a^\vark \psi_\a\|\le C
\end{equation*}
uniformly in $\a\le 1$.
\end{lem}

\begin{proof} By definition of $\psi_\a$,
\[
g_\a^\vark \psi_\a = \mu_\a^{-1} g_\a^\vark\BB_\a \psi_\a.
\]
In view of \eqref{Theta1:eq}, by definition \eqref{galpha:eq}
we have
$\T(x, y)\ge C|x|^\g\ge c \t(x)$, so that
the kernel $B_\a(x, y)$ is bounded from above by
\begin{equation*}
B_\a(x, y)\le \frac{\a}{\pi} \frac{C}{(x-y)^2+ \a^2 g_\a(x)^2},
\end{equation*}
and thus the kernel $\tilde B_\a(x, y) = g_\a(x)^\vark B_\a(x, y)$
satisfies the estimate
\[
\tilde B_\a(x, y)\le \frac{C}{\pi\a}
\frac{1}{\bigl(1+\a^{-2}(x-y)^2\bigr)^{1-\frac{\vark}{2}}}.
\]
Since $\vark < 1$, by Proposition \ref{schur:prop}
this kernel defines a bounded operator with
the norm uniformly bounded in $\a>0$. Thus
\[
\|g_\a^\vark \psi_\a\|\le C\mu_\a^{-1} \|\psi_\a\| \le C\mu_\a^{-1}.
\]
It remains to observe that by Lemma \ref{upper:lem} the eigenvalue
$\mu_\a$ is separated  from zero uniformly in $\a\le 1$.
\end{proof}

Now we obtain more delicate estimates for $\psi_\a$.
For a number $h\ge 0$ introduce the function
\begin{equation}\label{sh:eq}
S_\a(t; h) = \frac{\a}{\pi}\frac{1}{\a^2 + t^2 + h}, t\in\R,
\end{equation}
and denote by $\BS_\a(h)$ the integral operator with the kernel $\BS_\a(x-y; h)$. Along with
$\BS_\a(h)$ we also consider the operator
\[
\BT_\a(h) = \BS_\a(0) - \BS_\a(h).
\]
Due to \eqref{m:eq} the
Fourier transform of $S_\a(t; h)$ is
\begin{equation}\label{hats:eq}
\hat S_\a(\xi; h) = \frac{\a}{\sqrt{2\pi}\sqrt{\a^2+h}} e^{-|\xi|\sqrt{\a^2+h}},
\xi\in\R,
\end{equation}
 so that
 \begin{equation}\label{st_norm:eq}
 \|\BS_\a(h)\| = \frac{\a}{\sqrt{\a^2+h}},\ \
 \|\BT_\a(h)\| =  1 - \frac{\a}{\sqrt{\a^2+h}}.
 \end{equation}
Denote by $\chi_R$ the characteristic function of the interval $(-R, R)$.

\begin{lem} For sufficiently small $\a>0$ and $\a R\le 1$,
\begin{equation}\label{fourier_below:eq}
\| \hat\psi_\a \chi_R\|^2\ge 1 - \frac{4\l_1}{R}.
\end{equation}
\end{lem}

\begin{proof} Since $B_\a(x, y) < S_\a(x-y; 0)$ (see \eqref{B:eq} and \eqref{sh:eq}) and
$\psi_\a\ge 0$, we can write, using \eqref{hats:eq}:
\begin{align*}
\mu_\a  = &\ (\BB_{\a}\psi_\a, \psi_\a)< \int_\R\int_\R S_\a(x-y; 0)\psi_\a(x)\psi_\a(y) dxdy
= \int_{\R} e^{-\a|\xi|} |\hat\psi_\a(\xi)|^2 d\xi\\[0.2cm]
\le &\ \int_{|\xi|\le R} |\hat\psi_\a(\xi)|^2 d\xi +
e^{-\a R} \int_{|\xi|> R}|\hat\psi_\a(\xi)|^2 d\xi\\[0.2cm]
= &\ (1-e^{-\a R}) \int_{|\xi|\le R} |\hat\psi_\a(\xi)|^2 d\xi + e^{-\a R}.
\end{align*}
Due to \eqref{upper:eq}, $\mu_\a\ge 1-2\a\l_1$ for sufficiently small $\a$, so
\begin{equation*}
1-e^{-\a R} - 2\a\l_1\le (1-e^{-\a R}) \|\hat\psi_\a \chi_R\|^2,
\end{equation*}
which implies that
\begin{equation*}
\|\hat\psi_\a \chi_R\|^2\ge 1 - \frac{2\a\l_1}{1-e^{-\a R}}.
\end{equation*}
Since $e^{-s}\le (1+s)^{-1}$ for all $s\ge 0$, we get
$(1-e^{-s})^{-1}\le 2 s^{-1}$ for $0< s\le 1$, which entails
\eqref{fourier_below:eq} for $\a R\le 1$.
\end{proof}

\begin{lem}
For sufficiently small $\a>0$ and any $R>0$,
\begin{equation}\label{below:eq}
\|\psi_\a \chi_R\|
\ge 1 - 4\a \l_1 -  \frac{C} {R^{\g }},
\end{equation}
with some constant $C>0$ independent of $\a$ and $R$.
\end{lem}

\begin{proof}
It follows from \eqref{Theta1:eq} that $\T(x, y)\ge c|x|^\g$, so that the kernel
 $B_\a(x, y)$ satisfies the bound
\begin{equation*}
B_\a(x, y)\le S_\a(x - y; c\a^3 R^{\g}), \ \ \textup{for}\ \ \ |x|\ge R > 0.
\end{equation*}
Since $\psi_\a\ge 0$,
\begin{align*}
\mu_\a  = (\BB_{\a}\psi_\a, \psi_\a)
\le &\
(\BS_\a(0) \psi_\a, \psi_\a\chi_R)
+
\bigl(\BS_\a(c\a^3 R^{\g}) \psi_\a, \psi_\a(1-\chi_R)\bigr)\notag\\[0.2cm]
= &\ \bigl(\BT_\a(c\a^3 R^{\g})\psi_\a, \psi_\a\chi_R\bigr)
+ \bigl(\BS_\a(c\a^3 R^{\g}) \psi_\a, \psi_\a\bigr).
\end{align*}
In view of \eqref{st_norm:eq},
\begin{align*}
\mu_\a\le &\ \|\BT_\a(c\a^3 R^{\g})\|\ \|\psi_\a\chi_R\| + \|\BS_\a(c\a^3 R^{\g})\| \\[0.2cm]
= &\ \biggl(1-\frac{1}{\sqrt{1+c\a R^{\g}}}\biggr)\|\psi_\a\chi_R\|
+ \frac{1}{\sqrt{1+c\a R^{\g}}}.
\end{align*}
Using, as in the proof of the previous lemma,
the bound \eqref{upper:eq}, we obtain that
\begin{equation*}
 1 - \frac{1}{\sqrt{1+c\a R^{\g}}} - 2\a\l_1
 \le \biggl(1-\frac{1}{\sqrt{1+c\a R^{\g}}}\biggr)\|\psi_\a\chi_R\|,
 \end{equation*}
 so
 \begin{equation*}
 1 - \frac{4\l_1(1+c\a R^{\g})}{cR^{\g}}
 \le
\|\psi_\a \chi_R\|.
\end{equation*}
  This entails \eqref{below:eq}.
\end{proof}

Now we show that any sequence from
the family $\psi_\a$ contains a norm-convergent subsequence. The proof is inspired
by \cite{Widom}, Lemma 7.
We precede it with the following elementary result.

\begin{lem}\label{weak:lem}
Let $f_j\in\plainL2(\R)$ be a sequence such that $\|f_j\|\le C$ uniformly in
$j = 1,2, \dots$, and  $f_j(x) = 0$ for all $|x|\ge \rho >0$
and all $j = 1, 2, \dots$. Suppose that $f_j$
converges weakly to $f\in\plainL2(\R)$ as $j\to\infty$, and that for some constant $A>0$,
and all $R\ge R_0>0$,
\begin{equation}\label{local:eq}
\| \hat f_j\chi_R\|
\ge A - C R^{-\vark}, \ \vark >0,
\end{equation}
uniformly in $j$. Then $\|f\|\ge A$.
\end{lem}

\begin{proof}
Since $f_j$ are uniformly compactly supported, the Fourier transforms $\hat f_j(\xi)$
converge to $\hat f(\xi)$ a.a. $\xi\in\R^d$ as $j\to\infty$.
Moreover, the sequence $\hat f_j(\xi)$
is uniformly bounded, so $\hat f_j\chi_R\to \hat f \chi_R$, $j\to\infty$
in $\plainL2(\R)$ for any $R>0$.
Therefore \eqref{local:eq} implies that
\[
\| \hat f \chi_R\|\ge A - CR^{-\vark}.
\]
Since $R$ is arbitrary, we have $\| f\| = \|\hat f\|\ge A$, as claimed.
\end{proof}

\begin{lem}\label{strongly:lem}
For any sequence $\a_n\to 0, n\to\infty$, there exists a subsequence
$\a_{n_k}\to 0, k\to\infty$, such that the eigenfunctions $\psi_{\a_{n_k}}$
converge in norm as $k\to\infty$.
\end{lem}

\begin{proof}
Since the functions $\psi_\a, \a\ge 0$ are  normalized,
there is a subsequence $\psi_{\a_{n_k}}$ which converges weakly. Denote the limit
by $\psi$. From now on we write $\psi_k$ instead of $\psi_{\a_{n_k}}$
to avoid cumbersome notation. In view of the relations
\begin{equation*}
\| \psi_k-\psi\|^2 = 1 + \|\psi\|^2 - 2\re (\psi_k, \psi)\to
1 - \|\psi\|^2, k\to\infty,
\end{equation*}
it suffices to show that $\|\psi\| = 1$.

Fix a number $\rho>0$, and split $\psi_k$ in the following way:
\[
\psi_k(x) = \psi^{(1)}_{k, \rho}(x) + \psi^{(2)}_{k, \rho}(x),\
\psi^{(1)}_{k, \rho}(x) = \psi_k(x) \chi_\rho(x).
\]
Clearly, $\psi^{(1)}_{k, \rho}$ converges weakly to
$\psi_\rho = \psi\chi_{\rho}$ as $k\to\infty$.
Assume that $\a_{n_k}\le \rho^{-\g}$, so that by \eqref{below:eq},
\[
\|\psi^{(1)}_{k, \rho}\|^2 \ge 1 - \frac{C}{\rho^{\g}}, \ \
\|\psi^{(2)}_{k, \rho}\|^2 \le \frac{C}{\rho^{\g}}.
\]
Therefore, for any $R>0$,
\[
\|\widehat{\psi^{(1)}_{k, \rho}}\chi_R\|
\ge \|\hat \psi_k \chi_R\| - \|\psi^{(2)}_{k, \rho}\|
\ge 1- 4\l_1 R^{-1} - C\rho^{-\frac{\g}{2}},
\]
where we have used \eqref{fourier_below:eq}. By Lemma \ref{weak:lem},
\[
\|\psi_\rho\|\ge 1 - C\rho^{-\frac{\g}{2}}.
\]
Since $\rho$ is arbitrary, $\|\psi\|\ge 1$, and hence $\|\psi\|=1$. As a consequence,
the sequence $\psi_k$ converges in norm, as claimed.
\end{proof}

\section{Asymptotics of $\mu_\a, \a\to 0$: proof of Theorem \ref{main:thm}}

As before, by $\l_l$, $l = 1, 2, \dots$ we denote
the eigenvalues of $\BA$ arranged in ascending order,
and by $\phi_l$ -- a set of corresponding normalized eigenfunctions.
Recall that
the lowest eigenvalue $\l_1$ of the model operator $\BA$ is non-degenerate and
its (normalized) eigenfunction $\phi_1$ is chosen to be positive a.a. $x\in\R$.
We begin with proving Theorem \ref{main_b:thm}.

\begin{proof}[Proof of Theorem \ref{main_b:thm}]
The proof essentially follows the plan of \cite{Widom}.
It suffices to show that for
any sequence $\a_n\to 0, n\to \infty, $ one can find a subsequence
$\a_{n_k}\to 0$, $k\to \infty$ such that
\[
\lim_{k\to \infty} \a_{n_k}^{-1}(1-\mu_{\a_{n_k}}) = \l_1,
\]
and $\psi_{\a_{n_k}}$ converges in norm to $\phi_1$ as $k\to\infty$.
By Lemma \ref{strongly:lem} one can
pick a subsequence $\a_{n_k}$
such that  $\psi_{\a_{n_k}}$ converges in norm as $k\to \infty$.
As in the proof of Lemma \ref{strongly:lem} denote by $\psi$ the limit,
so $\|\psi\| = 1$ and $\psi\ge 0$ a.e..
For simplicity we write $\psi_\a$ instead of $\psi_{\a_{n_k}}$.
For an arbitrary function $f\in D(\BA)$ write
\begin{align*}
\mu_{\a}(\psi_\a, f)
= &\ (\BB_\a \psi_{\a}, f)
= (\psi_{\a}, \BB^{(l)}_\a f ) + (\psi_{\a}, (\BB_\a- \BB^{(l)}_\a) f )\\[0.2cm]
= &\ (\psi_{\a}, f)  - \a (\psi_{\a}, \BA f ) + (\psi_{\a}, \BR_\a f)
+ (\psi_{\a}, (\BB_\a- \BB^{(l)}_\a) f ).
\end{align*}
This implies that
\begin{equation}\label{mua:eq}
\a^{-1}(1-\mu_{\a})(\psi_{\a}, f)
= (\psi_{\a}, \BA f ) -\a^{-1} (\psi_{\a}, \BR_\a f)
- \a^{-1} (\psi_{\a}, (\BB_\a- \BB^{(l)}_\a) f ).
\end{equation}
In view of \eqref{k-l:eq} the last term on the right-hand side
tends to zero as $\a\to 0$. The
first term trivially tends to $(\psi, \BA f)$.
Consider the second term:
\begin{align*}
|(\psi_{\a}, \BR_\a f)|
= &\ (\psi_\a, \BR^{(1)}_\a f)
+ (g_\a^\vark \psi_\a, \BE^{(2)}_\a\lu D_x\ru \z\bigl((\a\lu D_x\ru)^\vark\bigr)f)\\[0.2cm]
\le &\ \|\BR^{(1)}_\a f\| + \|g_\a^\vark \psi_\a\|
\ \|\BE^{(2)}_\a\lu D_x\ru \z\bigl((\a\lu D_x\ru)^\vark\bigr)f\|.
\end{align*}
Assume now that $\vark <1$.
By Corollary \ref{g:cor} and  Lemma \ref{psi:lem}, the right-hand side
is of order $o(\a)$, and hence, if
$(\psi, f)\not = 0$, then passing to the limit in \eqref{mua:eq} we get
\[
\lim_{\a\to 0} \a^{-1}(1-\mu_{\a})   = \frac{(\psi, \BA f)}{(\psi, f)}.
\]
Let $f = \phi_l$ with some $l$, so that
$(\psi, \BA f) = \l_l(\psi, \phi_l)$.
Suppose that $(\psi, \phi_l)\not = 0$, so that
\[
\lim_{\a\to 0} \a^{-1}(1-\mu_{\a}) = \l_l.
\]
By the
uniqueness of the above limit,
$(\psi, \phi_j) = 0$ for all $j$'s such that $\l_j\not=\l_k$.
Thus, by completeness of the system $\{\phi_k\}$,
the function $\psi$  is an eigenfunction
of $\BA$ with the eigenvalue $\l_l$.
In view of \eqref{upper:eq}, $\l_l\le \l_1$. Since the eigenvalues $\l_j$
are labeled in ascending order we conclude that $\l_l = \l_1$.
As this eigenvalue is non-degenerate and the corresponding eigenfunction
$\phi_1$ is positive a.e., we observe that $\psi = \phi_1$.
\end{proof}

\begin{proof}[Proof of Theorem \ref{main:thm}]
Theorem \ref{main:thm} follows from Theorem \ref{main_b:thm}
due to the relations \eqref{unitary:eq}.
\end{proof}

\section{Miscellaneous}

In this short section we collect some open questions related to the spectrum of the operator
\eqref{bkb:eq}.

\subsection{}
Theorems \ref{main:thm} and \ref{main_b:thm}  give information on the largest eigenvalue
$\SM_\beta$ of the operator $\BK_\beta$ defined in \eqref{bkb:eq}, \eqref{kernel:eq}. Let
\begin{equation}\label{mbj:eq}
\SM_\beta \equiv \SM_{1, \b} \geq \SM_{2, \b} \geq \ldots
\end{equation}
be the sequence of all positive eigenvalues of $\BK_\b$ arranged in
descending order.
The following conjecture is a natural extension of Theorem \ref{main:thm}.

\begin{con}\label{extremal:conj}
For any $j = 1, 2, ...$
\begin{equation}\label{mjb_ass:eq}
\lim_{\b \to 0} \b^{- \frac{2}{\g + 1}}
(1 - \SM_{j, \b}) = \l_j,
\end{equation}
where $\l_1 < \l_2 \leq \ldots $
are eigenvalues of the operator $\BA$ defined in \eqref{modela:eq},
arranged in ascending order.
\end{con}

For the case $\T(x, y) = (x^2 + y^2)^2$ the formula \eqref{mjb_ass:eq}
was conjectured in \cite{M}, Section 7.1, but
without specifying what the values $\l_j$ are. As in \cite{M}, the formula
\eqref{mjb_ass:eq} is prompted by the paper \cite{Widom} where asymptotics of
the form \eqref{mjb_ass:eq} were found for an integral operator with a difference kernel.

\subsection{}
Although the operator $\BK_\b$ converges strongly to
the positive-definite operator $\BK_0$ as $\b\to 0$, we can't say whether or
not $\BK_\b, \b >0,$  has negative eigenvalues.

\subsection{}
Suppose that the function $\T(x, y)$ in \eqref{kernel:eq} is even, i.e.
$\T(-x, -y) = \T(x, y)$, $x, y\in\R$. Then the
subspaces $H^{\textup{\tiny e}}$
and $H^{\textup{\tiny o}}$ in $\plainL2(\R)$ of even and odd
functions are invariant for $\BK = \BK_\beta$.
Consider restriction operators
$\BK^{\textup{\tiny e}} = \BK\upharpoonright H^{\textup{\tiny e}}$
and $\BK^{\textup{\tiny o}} = \BK\upharpoonright H^{\textup{\tiny o}}$
and their positive eigenvalues $\l^{\textup{\tiny e}}_j$ and $\l^{\textup{\tiny o}}_j$,
$j = 1, 2, \dots$, arranged in descending order.
Remembering that the top eigenvalue of $\BK$ is non-degenerate and its eigenfunction
is positive a.e., one easily concludes that
$\l^{\textup{\tiny e}}_1 > \l^{\textup{\tiny o}}_1$.
Are there  similar inequalities for the pairs
$\l^{\textup{\tiny e}}_j, \l^{\textup{\tiny o}}_j$ with $j > 1$?

\section{Appendix. Boundedness of integral and pseudo-differential operators}

In this Appendix, for the reader's convenience we remind (without proofs) simple tests
of boundedness for integral and pseudo-differential operators acting on $\plainL2(\R^d)$,
$d\ge 1$.  Consider the integral operator
\begin{equation}\label{integral:eq}
(K u)(\bx) =  \int_{\R^d}  K(\bx, \by) u(\by) d\by,
\end{equation}
with the kernel $K(\bx, \by)$, and the pseudo-differential operator
\begin{equation}\label{pdo:eq}
(\op(a) u)(\bx) = \frac{1}{(2\pi)^d} \int_{\R^d}\int_{\R^d}
e^{i(\bx-\by)\cdot\bxi} a(\bx, \bxi) u(\by) d\by \bxi,
\end{equation}
with the symbol $a(\bx, \bxi)$.

The following classical result is
known as the Schur Test and it can be found, even in a more general
form, in \cite{HalS}, Theorem 5.2.

\begin{prop}\label{schur:prop}
Suppose that the kernel $K$ satisfies the conditions
\begin{equation*}
M_1 = \sup_{\bx}\int_{\R^d} |K(\bx, \by)| d\by < \infty,\ \
M_2 = \sup_{\by}\int_{\R^d} |K(\bx, \by)| d\bx < \infty.
\end{equation*}
Then the operator \eqref{integral:eq} is bounded on $\plainL2(\R^d)$
and $\|K\|\le \sqrt{M_1 M_2}$.
\end{prop}

For pseudo-differential operators
on $\plainL2(\R^d)$ we use the test of boundedness
found by H.O.Cordes in \cite{Cordes}, Theorem $B_1'$.

\begin{prop}\label{boundedness:prop}
Let $a(\bx, \bxi), \bx, \bxi\in\R^d, d\ge 1$, be a function
such that its distributional derivatives of the form
$\nabla^n_{\bx}\nabla^m_{\bxi} a$
are $\plainL\infty$-functions for all $0\le n, m\le r$, where
\[
r = \left[\frac{d}{2}\right]+1.
\]
Then the operator \eqref{pdo:eq} is bounded on $\plainL2(\R^d)$ and
\begin{equation*}
\|\op( a)\|\le C \max_{0\le n, m\le r}
\| \nabla_{\bx}^{n} \nabla_{\bxi}^m a\|_{\plainL\infty},
\end{equation*}
with a constant $C$ depending only on $d$.
\end{prop}

It is important for us that for $d=1$ the above test requires
the boundedness of derivatives $\p_x^n \p_\xi^m a$ with $n, m\in\{0, 1\}$
only. This result is extended to arbitrary dimensions by M. Ruzhansky and M. Sugimoto,
see \cite{RuzhSug} Corollary 2.4.
Recall that the classical Calder\'on-Vaillancourt theorem needs
more derivatives with respect to each variable, see \cite{Cordes}
and \cite{RuzhSug} for discussion. A short prove of  Proposition \ref{boundedness:prop}
was given by I.L. Hwang in \cite{Hw}, Theorem 2 (see also \cite{Lerner}, Lemma 2.3.2
for a somewhat simplified version).

\bibliographystyle{amsplain}
\bibliography{bibmaster}


\end{document}